\documentclass[reqno]{amsart}
\usepackage{mathrsfs,amssymb,mathrsfs, multirow,setspace}
\usepackage{amsmath}
\usepackage[a4paper, total={7in, 9in}]{geometry}
\usepackage{enumerate}
\theoremstyle{plain}
\usepackage{graphicx}
\usepackage{mathtools}
\usepackage[utf8]{inputenc}

\newtheorem{theorem}{Theorem}[section]
\newtheorem{lemma}[theorem]{Lemma}
\newtheorem{proposition}[theorem]{Proposition}
\newtheorem{corollary}[theorem]{Corollary}

\theoremstyle{definition}

\newtheorem{example}[theorem]{Example}

\theoremstyle{remark}

\input xy
\xyoption{all}

\renewcommand{\Bbb}{\mathbb} 

\begin{document}
	
	\title[Sombor index and eigenvalues of  weakly zero-divisor graph of commutative rings]{Sombor index and eigenvalues of weakly zero-divisor graph of commutative rings}
  \author[Mohd Shariq, Jitender Kumar*]{ Mohd Shariq, Jitender Kumar*}
   \address{Department of Mathematics, Birla Institute of Technology and Science Pilani, Pilani-333031, India}
  \email{shariqamu90@gmail.com,  jitenderarora09@gmail.com}


\keywords{Sombor matrix, Sombor index, Sombor spectrum, Weakly zero-divisor graph, eigenvalues, ring of integers modulo $n$,\\ *  Corresponding author}
\maketitle



 \begin{abstract}
      The weakly zero-divisor graph $W\Gamma(R)$ of a commutative ring $R$ is the simple undirected graph whose vertices are nonzero zero-divisors of $R$ and two distinct vertices $x$, $y$ are adjacent if and only if there exist $w\in {\rm ann}(x)$ and $ z\in {\rm ann}(y)$ such that $wz =0$. In this paper, we determine the Sombor index for the weakly zero-divisor graph of the integers modulo ring $\mathbb{Z}_n$. Furthermore, we investigate the Sombor spectrum and establish bounds for the Sombor energy of the weakly zero-divisor graph of $\mathbb{Z}_n$.
   \end{abstract}

 \section{Introduction} 
 Let $G$ be a simple graph with vertex set $ V(G)$ and edge set $ E(G)$. The number of vertices and edges in $G$ are called the \emph{order} and \emph{size} of the graph $G$, respectively. Two adjacent vertices $u$ and $v$ are denoted by $u\sim v$. The \emph{neighbourhood} $N_G(v)$ of a vertex $v \in V$ is defined as the set of vertices adjacent to $v$. The \emph{degree} $\frak{d}_{v}$ of a vertex $v$ is the number of elements in the set $N_G(v)$. The \emph{complement} $\overline{G}$ of the graph $G$ is a graph with the same vertex set as $G$ and two distinct vertices $u, v$ are adjacent in $\overline{G}$ if and only if they are not adjacent in $G$. The \emph{complete graph} on $n$ vertices is denoted by $K_n$. An \emph{independent set} in a graph $G$ is a subset of vertices $S \subseteq V(G)$ such that no two vertices in $S$ are adjacent. A graph $G$ is said to be \emph{$k$-regular} if $\frak{d_v} =k $ for every vertex $v$ in $G$. 
 The \emph{join} $G_1 \vee G_2$ of graphs $G_1$ and $G_2$ with vertex set $V_1 \cup V_2$ is the graph in which each vertex of $V_1$ is adjacent to every vertex of $V_2$. For additional notations and terminology, readers are referred to \cite{brouwer2011spectra,horn2012matrix,nicholson2012introduction}.
  
 In both mathematical and chemical literature, numerous well-known degree-based graph invariants, often referred to as topological indices, have been extensively studied. Recently, Gutman \cite{gutman2021geometric} introduced a new topological index, called the \textit{Sombor index}, denoted by $ \text{SO}(G) $, which is defined as:

\[
\text{SO}(G) = \sum_{\{v_i, v_j\} \in E(G)} \sqrt{\frak{d}_{v
_i}^2 + \frak{d}_{v_j}^2}
,\]
where $ \frak{d}_{v_i} $ and $ \frak{d}_{v_j} $ represent the degrees of the vertices $ v_i $ and $ v_j $ in the graph $ G $ and the summation runs over all edges in $ E(G) $. Das et al. \cite{das2021sombor} established novel bounds for the Sombor index $SO(G)$ and explored its connections with other topological indices, such as the Zagreb indices of $G$. Do{\v{s}}lic et al. \cite{reti2021sombor} examined the extremal properties of the Sombor index of various graph families, established that the cycle $C_n$ attains the minimum Sombor index among all unicyclic graphs.
The Sombor index has been examined by several researchers for its chemical applications and molecular properties, see \cite{deng2021molecular, liu2022sombor, redvzepovic2021chemical} and references therein. The Sombor index is important not only in combinatorics and linear algebra but also has a deeper geometric meaning, see \cite{gutman2021geometric, gutman2022sombor} and references therein.

The \emph{Sombor matrix} of a graph $G$, denoted by $S(G)$ is defined as:
\[
S(G) = (s_{ij}) = 
\begin{cases} 
\sqrt{\frak{d}_{v_i}^2 + \frak{d}_{v_j}^2}, & \text{if } v_i \text{ and } v_j \text{ are adjacent}, \\
0, & \text{otherwise}.
\end{cases}
\]
 Clearly, this matrix is both real and symmetric. Let the eigenvalues of $S(G)$ be denoted by $\frak{\lambda_i}$, ordered as $\frak{\lambda_1} \geq \frak{\lambda_2} \geq \cdots \geq \frak{\lambda_n}$. The collection of all eigenvalues (counting multiplicities) of $S(G)$ is referred to as the \emph{Sombor spectrum} of $G$. The largest eigenvalue of $S(G)$ is known as the \emph{Sombor spectral radius} of $G$, and for real symmetric matrices, it coincides with the spectral norm of $G$. 
 The \emph{Sombor energy} of $G$ is defined as:
\[
ESO(G) = \sum_{i=1}^{n} |\frak{\lambda_i}|.
\]
 Several papers have explored the spectral properties of the Sombor matrix, such as the characteristics of Sombor eigenvalues, spectral radius, energy, Estrada index and the relationships between energy, Sombor energy, the Sombor index and other invariants (see, \cite{ghanbari2022sombor,ivan2021spectrum,pirzada2021sombor,rather2022sharp}).
 
To expand the class of zero-divisor graphs, the notion of weakly zero-divisor graph has been introduced by Nikmehr et al.  \cite{nikmehr2021weakly}. The \emph{weakly zero-divisor graph} $W\Gamma(R)$ of the ring $R$ is the simple undirected graph whose vertex set is the set of all nonzero zero-divisors of $R$  and two vertices $x, y$ are adjacent if and only if there exist $w\in {\rm ann}(x)$ and $z \in {\rm ann}(y)$ such that $wz =0$. One can observe that the zero-divisor graph is a spanning subgraph of the weakly zero-divisor graph. Nikmehr et al.  \cite{nikmehr2021weakly} explored the relationship between zero-divisor graphs and weakly zero-divisor graphs. Moreover, they examined several basic properties, including completeness, girth, clique number, and vertex chromatic number of $W\Gamma(R)$.
Rehman et al. \cite{ur2024planarity} classified all the commutative rings $R$ for which the weakly zero-divisor graph is star graphs, unicyclic graphs, trees and split graphs. Furthermore, they identified all the rings $R$ for which $W\Gamma(R)$ is planar, toroidal, bitoroidal and of crosscap almost two, respectively. Many other researchers have studied and investigated weakly zero-divisor graphs of commutative rings from different perspectives (see, \cite{nazim2023normalized,rehman2023analysis,rehman2024randic,shariq2023laplacianspectrum}).

The spectral properties and topological indices are very well studied for graphs defined on rings; see \cite{Asirwiner,bajaj2022adjacency, Sriparnazero,pirzada2020distance,ShenShouqiang2022} and references therein. Kumar et al. \cite{shariq2023laplacianspectrum} investigated the Laplacian eigenvalues of the weakly zero-divisor graph of the ring $\mathbb{Z}_n$. The Randic spectrum and normalized Laplacian
 spectrum of the weakly zero-divisor graph of the ring $\mathbb{Z}_n$ has been determined by Nadeem et al. \cite{nazim2023normalized} \cite{rehman2024randic}. Rashid et al. \cite{mozumder2025exploring} explored the $A_\alpha$ spectrum of the weakly zero-divisor graph of the ring $\mathbb{Z}_n$. Gursoy et al. \cite{gursoy2022sombor} investigated the Sombor index of the zero-divisor graph of the integer modulo ring $\mathbb{Z}_n$ for different values of $n$. Bilal et al. \cite{rather2024sombor} established the upper and lower bounds of the Sombor index for comaximal graphs of the ring $\mathbb{Z}_n$. Additionally, they determined the Sombor eigenvalues and provided the bounds for the Sombor energy of the comaximal graphs.
Anwar et al. \cite{anwar2024sombor} investigated the Sombor index of the cozero-divisor graph of the ring $\mathbb{Z}_n$ for various values of $n$. Moreover, they determined the Sombor spectrum of the cozero-divisor graph. Motivated by the above-cited work, in this manuscript, we aim to study the Sombor index and the Sombor spectrum of the weakly zero-divisor graph of the ring $\mathbb{Z}_n$. The rest of the manuscript is organized as follows: Section 2, recall the necessary definitions and results. In section 3, we investigate the Sombor index of the weakly zero-divisor graph $W(\Gamma(\mathbb{Z}_n))$. For arbitrary $n$, Section 4 determines the Sombor eigenvalues of $W(\Gamma(\mathbb{Z}_n))$ and establishes the bounds for Sombor energy of $W(\Gamma(\mathbb{Z}_n))$.

\section{sombor index  of the weakly zero-divisor graph integers modulo ring} 
Let $\mathbb{Z}_n = \{0, 1, \ldots, n-1\}$ be the the ring of integers modulo $n$. The number of integers less than $n$ that are coprime to $n$ is given by the \emph{Euler totient function}, denoted as $\phi(n)$. An integer $d$ is called a \emph{proper divisor} of $n$ if $1 < d < n$ and $d \mid n$. The total number of divisors of $n$ is represented by $\tau(n)$. The greatest common divisor of positive integers $a$ and $b$ is denoted by $\text{gcd}(a, b)$.
In the ring $\mathbb{Z}_n$, the ideal generated by an element $a$ is defined as the set $\{xa: x \in \mathbb{Z}_n\}$ and is denoted by $\langle a \rangle$.
 Let $d_1, d_2, \ldots, d_k$ be  the proper divisors of $n$. For $1 \leq i \leq k$, consider the following sets:
\[
\mathcal{A}_{d_i} = \{x \in \mathbb{Z}_n : \text{gcd}(x, n) = d_i\}.
\]
Also, observe that $\mathcal{A}_{d_{i}}\cap \mathcal{A}_{d_{j}} = \varnothing$, when $i \neq j$. It follows that the sets $\mathcal{A}_{d_{1}}, \mathcal{A}_{d_{2}}, \dots,\mathcal{A}_{d_{k}}$ are pairwise disjoint and form a partition of the vertex set of the graph  $W\Gamma(\mathbb{Z}_n)$. Therefore,
\[V(W\Gamma(\mathbb{Z}_n)) = \mathcal{A}_{d_{1}} \cup \mathcal{A}_{d_{2}} \cup \cdots \cup \mathcal{A}_{d_{k}}.\] 
The cardinality of each $\mathcal{A}_{d_i}$ is known in the following lemma. 
\begin{lemma}\cite{MR3404655}\label{valueof partition}
 For $1 \leq i \leq k$, we have
$|\mathcal{A}_{d_i}| = \phi\left(\frac{n}{d_i}\right)$. 
\end{lemma}

 \begin{lemma}\cite[Lemma 3.3]{shariq2023laplacianspectrum}\label{adjacenyofvertex}
 Let $x \in {\mathcal{A}_{d_i}}$, $ y \in \mathcal{A}_{d_j}$, where $i, j \in \{1, 2, \ldots, k\}$ and $i\neq j$. Then $x \sim y$ in $W\Gamma (\mathbb{Z}_n)$.
 \end{lemma}

\begin{lemma}\cite[Lemma 3.5]{shariq2023laplacianspectrum}\label{adjacenyofvertexinjoin} 
For any proper divisor  $d_i$ of $n$, $W\Gamma (\mathcal{A}_{d_i})$ is either a null graph or a complete graph.
\end{lemma}

If $n={p_1}^{k_1}{p_2}^{k_2}\ldots{p_m}^{k_m},$ where $k_i\geq 2$, then by  \rm \cite[Theorem 2.6]{nikmehr2021weakly}, the graph $W\Gamma(\mathbb{Z}_n)$ is a complete. Also, if $n = p$, then the graph $W\Gamma(\mathbb{Z}_p)$ is a null graph. Therefore, in the remaining paper, we determine the Sombor spectrum of $W\Gamma(\mathbb{Z}_n)$ for $n = p_1p_2\ldots p_m{q_1}^{k_1}{q_2}^{k_2}\ldots{q_r}^{k_r} $, where $k_i\geq 2,m\geq2$ and $m+r\geq2$. 
 
\begin{lemma}\cite[Lemma 3.6]{shariq2023laplacianspectrum}\label{adjacenyofjoin}
Let $D$= $\{ d_1,d_2,\ldots, d_{k}\}$ be the set of all proper divisors of $n$ and let $n$= $p_1p_2\ldots$ $ p_m{q_1}^{k_1}{q_2}^{k_2}\ldots{q_r}^{k_r} $, where $k_i\geq 2,m\geq 2$ and $m+r\geq 2 $. Then the subgraph of $W\Gamma ({\mathbb Z_n}) $  induced by $\mathcal{A}_{d_i}$ is $\overline{K}_{\phi\left({\frac{n}{d_i}}\right)}$ if and only if $d_i\in \{p_1,p_2,\ldots ,p_m\}$.
\end{lemma}

\begin{corollary}\cite[Lemma 3.7]{shariq2023laplacianspectrum}\label{partitionofcozerodivisorgraphisomorphic}
The following statements hold:
\begin{itemize}
    \item[(i)] For $i \in \{1, 2, \ldots, k\}$, the  subgraph of $W\Gamma(\mathbb{Z}_n)$ induced by $\mathcal{A}_{d_i}$ is isomorphic to either $\overline{K}_{\phi\left(\frac{n}{d_i}\right)}$ or ${K}_{\phi\left(\frac{n}{d_i}\right)}.$
    \item [(ii)] For $i,j \in \{1, 2, \ldots, k\}$ and $i \neq j$, a vertex of $\mathcal{A}_{d_i}$ is adjacent to  all the vertices of $\mathcal{A}_{d_j}$. 
\end{itemize}
\end{corollary}
Let $G$ be a simple undirected graph on $k$ vertices with $V(G) = \{v_1, v_2, \dots, v_k\}$ and let $H_1, H_2, \dots, H_k$ be $k$ pairwise disjoint graphs. The \emph{generalized join graph} $G[H_1, H_2, \dots, H_k]$ is obtained by replacing each vertex $v_i$ of $G$ with the graph $H_i$ and then connecting each vertex of $H_i$ to every vertex of $H_j$ whenever $v_i \sim v_j$ in $G$. Note that if $G$ consists of only two adjacent vertices, then the \emph{generalised join graph} $G[H_1, H_2]$ is the same as the usual join $H_1 \vee H_2$ of $H_1$ and $H_2$.
 The following lemma states that $W(\Gamma(\mathbb{Z}_n))$  can be expressed as the generalised join of certain complete graphs and null graphs. The graph $ \Upsilon_n $ is defined as the simple graph whose vertices are the proper divisors $ d_1, d_2, \ldots, d_k $ of $ n $ and any two distinct vertices $ d_i $ and $ d_j $ are always adjacent.
\begin{lemma}\cite[Lemma 3.8]{shariq2023laplacianspectrum}\label{inducedsubgraphequaltogamma} 
Let  $W\Gamma(\mathcal{A}_{d_i})$ be the subgraph of $W\Gamma(\mathbb{Z}_{n})$ induced by $\mathcal{A}_{d_i}$ $(1\leq i\leq k)$. Then \[W\Gamma(\mathbb{Z}_n) = \Upsilon_n [W\Gamma(\mathcal{A}_{d_1}), W\Gamma (\mathcal{A}_{d_2}), \ldots, W\Gamma(\mathcal{A}_{d_k})].\]
\end{lemma}
\begin{example}
The weakly zero-divisor graph $W\Gamma(\mathbb{Z}_{18})$ is shown in Figure $1$. By Lemma \ref{inducedsubgraphequaltogamma}, we have  $W\Gamma(\mathbb{Z}_{18}) = \Upsilon_{18} [W\Gamma(\mathcal{A}_{2})$, 
$W\Gamma(\mathcal{A}_{3})$, $W\Gamma(\mathcal{A}_{6})$, 
$W\Gamma(\mathcal{A}_{9})$], where $\Upsilon_{18}$ is complete graph on the set $\{2,3,6,9\}$ and  $W\Gamma(\mathcal{A}_{2})$ = $\overline{K}_6$, $W\Gamma(\mathcal{A}_{3})$ = ${K}_2$ = $W\Gamma(\mathcal{A}_{6})$, $W\Gamma(\mathcal{A}_{9})$ = ${K}_1 $. 
\begin{figure}[h!]
\centering
\includegraphics[width=0.43 \textwidth]{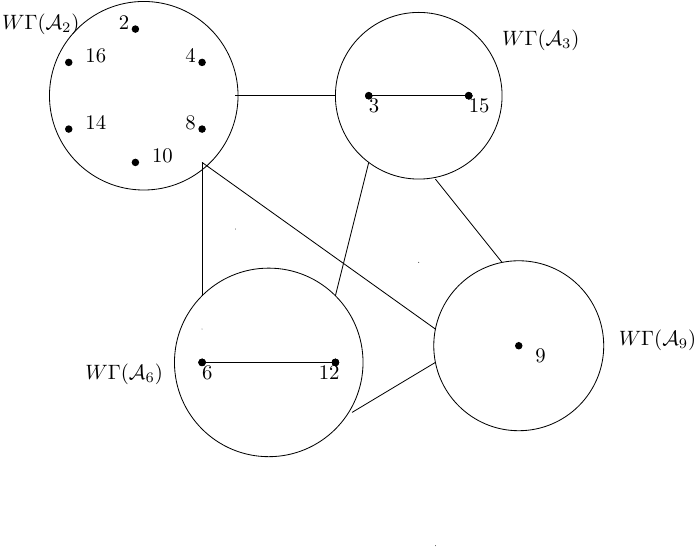}
\caption{The graph $W\Gamma(\mathbb{Z}_{18})$}
\end{figure}
\end{example}
\begin{theorem} Let $W\Gamma(\mathbb{Z}_n)$ be the weakly zero-divisor graph of the ring $\mathbb{Z}_n$. Then the following holds: 
\begin{itemize}
      \item[(i)] If $n={p_1}^{k_1}{p_2}^{k_2}\cdots{p_m}^{k_m},$ where $k_i\geq 2$, then 
 \begin{align*}
         SO(W\Gamma(\mathbb{Z}_n))&= \frac{(n-\phi(n)-1)}{\sqrt{2}}\left( n-\phi(n)-2\right)^2.
\end{align*}
    \item[(ii)] If $n = p_1p_2\cdots p_m{q_1}^{k_1}{q_2}^{k_2}\cdots{q_r}^{k_r} $, where $k_i\geq 2,m\geq 1$ and $m+r\geq 2 $, then \begin{align*}
         SO(W\Gamma(\mathbb{Z}_n))&= \frac{(n-\phi(n)-2)}{\sqrt{2}}\left( (n-\phi(n)-\sum \limits_{i=1 }^k\phi\left( \frac{n}{p_i}\right)-1\right)\left( n-\phi(n)-\sum \limits_{i=1 }^k\phi\left( \frac{n}{p_i}\right)-2\right)\\ &+\left( n-\phi(n)-\sum \limits_{i=1 }^k\phi\left( \frac{n}{p_i}\right)-1\right)\sum \limits_{i=1 }^k\phi\left( \frac{n}{p_i}\right)\sqrt{ \left(n-\phi(n)-2 \right)^2+\left(n-\phi(n)-\phi\left(\frac{n}{p_i}\right)-1\right)^2}  \\ & +  \sum \limits_{i=1 }^k\phi\left( \frac{n}{p_i}\right) \sum \limits_{i\neq j=1 }\phi\left(\frac{n}{p_j}\right) \sqrt{\left(n-\phi(n)-\phi\left(\frac{n}{p_i}\right)-1\right)^2 +\left(n-\phi(n)-\phi\left(\frac{n}{p_j}\right)-1\right)^2}.
     \end{align*}
\end{itemize}
    \end{theorem}
\begin{proof}
\begin{itemize}
    \item[(i)] Let  $n={p_1}^{k_1}{p_2}^{k_2}\cdots{p_m}^{k_m},$ where $k_i\geq 2$. Then by \cite[Theorem 2.6]{nikmehr2021weakly}, the graph $W\Gamma(\mathbb{Z}_n)$ is complete on $n-\phi(n)-1$ vertices. Therefore, the sombor index of $W\Gamma(\mathbb{Z}_n)$ is
    \begin{align*}
         SO(W\Gamma(\mathbb{Z}_n))&= \frac{(n-\phi(n)-1)}{\sqrt{2}}\left( n-\phi(n)-2\right)^2.
\end{align*}
\item[(ii)]   For $n = p_1p_2\cdots p_m{q_1}^{k_1}{q_2}^{k_2}\cdots{q_r}^{k_r}$ the weakly zero-divisor graph  of $\mathbb{Z}_n$ can be constructed as follows. Let  $D=\{ d_1,d_2\cdots d_{\tau(n)-2}\}$ be the set of all proper divisors of $n$. 
Then, by Lemma \ref{inducedsubgraphequaltogamma}, the weakly zero-divisor graph of $\mathbb{Z}_n$ is given by 
\[ W\Gamma(\mathbb{Z}_n) = \Upsilon_n [W\Gamma(\mathcal{A}_{d_1}), W\Gamma (\mathcal{A}_{d_2}), \ldots, W\Gamma(\mathcal{A}_{d_m}),  W\Gamma(\mathcal{A}_{d_{m+1}}), W\Gamma (\mathcal{A}_{d_{m+2}}), \ldots, W\Gamma(\mathcal{A}_{d_{\tau(n)-2}})].\]
Since for each $d_i,d_j \in D$,  $d_i$  is adjacent to $d_j$ by Lemma \ref{adjacenyofvertex}, it follows that the graph $\Upsilon_{n}$ is complete on the vertex set $D$.
Now, let $A=\{p_1,p_2,\ldots p_m\}$. Then by Lemma \ref{adjacenyofjoin}, for each $d_i\in A$, we have $W\Gamma(\mathcal{A}_{d_{i}}) = \overline{K}_{\phi \left(\frac{n}{d_i}\right)}$ and $W\Gamma(\mathcal{A}_{d_j}) =  {K}_{\phi\left(\frac{n}{d_j}\right)}$, where $d_j\notin A$. Consequently, 
\begin{align*}  W\Gamma(\mathbb{Z}_n)
 &= \Upsilon_n [\overline{K}_{\phi\left(\frac{n}{p_1}\right)}, \overline{K}_{\phi\left(\frac{n}{p_2}\right)},\ldots, \overline{K}_{\phi\left(\frac{n}{p_m}\right)}, {K}_{\phi\left(\frac{n}{d_{m+1}}\right)},{K}_{\phi\left(\frac{n}{d_{m+2}}\right)},\ldots,{K}_{\phi\left(\frac{n}{d_{\tau(n)-2}}\right)}]\\
 &= K_{\left(n-\phi(n)-\sum\limits_{i=1}^k\phi\left( \frac{n}{p_i}\right)-1\right)} \vee K_m[\overline{K}_{\phi\left(\frac{n}{p_1}\right)}, \overline{K}_{\phi\left(\frac{n}{p_2}\right)},\ldots, \overline{K}_{\phi\left(\frac{n}{p_m}\right)}].
     \end{align*}
     where $K_m$ is complete graph on the vertices $ p_1,p_2,\ldots,p_m$.

Let $\{v_1,v_2,\ldots,v_{\left(n-\phi(n)-\sum\limits_{i=1}^m\phi \left(\frac{n}{p_i}\right)-1\right)}, u_{11},  u_{12},\ldots, u_{1\phi\left(\frac{n}{p_1}\right)}, u_{21},u_{22},\ldots, u_{2\phi\left(\frac{n}{p_2}\right)},\ldots, u_{m1},u_{m2},\ldots, u_{m\phi\left(\frac{n}{p_m}\right)}\}$ be the vertex labeling of $W\Gamma(\mathbb{Z}_n)$, where $v_i$’s are vertices of 
$K_{\left(n-\phi(n)-\sum\limits_{i=1}^m\phi \left(\frac{n}{p_i}\right)-1\right)}$, $u_{ij}$
are the vertices of $\overline{K}_{\phi\left(\frac{n}{p_j}\right)}$ for $i= 1,2,\ldots,m$ and $j= 1,2,\ldots,m$.
Observe that $\frak{d}(v_i)= n-\phi(n)-2$. Since $K_m$  is complete graph on the vertices $p_1,p_2,\ldots,p_m$, it follows that $\frak{d}(u_{ij})= n-\phi(n)-\sum\limits_{i=1}^m\phi \left(\frac{n}{p_i}\right)-1+ \sum\limits_{i\neq j=1}^m\phi \left(\frac{n}{p_j}\right)=n-\phi(n)-\phi\left(\frac{n}{p_i}\right)-1$. 
Therefore, the Sombor index of $G$ is
\begin{align*}
    SO(W\Gamma(\mathbb{Z}_n))&=\sum \limits_{v_i\sim v_{j}}\sqrt{{\frak{d}_{v_i}}^2+{\frak{d}_{v_j}}^2}+\sum \limits_{v_i\sim u_{kl}}\sqrt{{\frak{d}_{v_i}}^2+{\frak{d}_{u_{ik}}}^2} +\sum \limits_{u_{ij}\sim u_{kl}}\sqrt{{\frak{d}_{u_{ik}}}^2+{\frak{d}_{v_{jl}}}^2}\\&= \frac{1}{2}\left(n-\phi(n)-\sum\limits_{i=1}^m\phi \left(\frac{n}{p_i}\right)-1\right) \left(n-\phi(n)-\sum\limits_{i=1}^m\phi \left(\frac{n}{p_i}\right)-2\right)\sqrt{(n-\phi(n)-2)^2+(n-\phi(n)-2)^2}  \\& + \left(n-\phi(n)-\sum\limits_{i=1}^m\phi \left(\frac{n}{p_i}\right)-1\right)\sum\limits_{i=1}^m\phi \left(\frac{n}{p_i}\right)\sqrt{(n-\phi(n)-2)^2+ \left(n-\phi(n)-\phi\left(\frac{n}{p_i}\right)-1\right)^2} \\&+ \sum\limits_{i=1}^m\phi \left(\frac{n}{p_i}\right) \sum\limits_{i\neq j=1}^m\phi \left(\frac{n}{p_j}\right) \sqrt{\left(n-\phi(n)-\phi\left(\frac{n}{p_i}\right)-1\right)^2+\left(n-\phi(n)-\phi\left(\frac{n}{p_j}\right)-1\right)^2}
    \\&=\frac{(n-\phi(n)-2)}{\sqrt{2}} \left(n-\phi(n)-\sum\limits_{i=1}^m\phi \left(\frac{n}{p_i}\right)-1\right) \left(n-\phi(n)-\sum\limits_{i=1}^m\phi \left(\frac{n}{p_i}\right)-2\right)  \\& + \left(n-\phi(n)-\sum\limits_{i=1}^m\phi \left(\frac{n}{p_i}\right)-1\right)\sum\limits_{i=1}^m\phi \left(\frac{n}{p_i}\right)\sqrt{(n-\phi(n)-2)^2+ \left(n-\phi(n)-\phi\left(\frac{n}{p_i}\right)-1\right)^2} \\& + \sum\limits_{i=1}^m\phi \left(\frac{n}{p_i}\right) \sum\limits_{i\neq j=1}^m\phi \left(\frac{n}{p_j}\right) \sqrt{\left(n-\phi(n)-\phi\left(\frac{n}{p_i}\right)-1\right)^2+\left(n-\phi(n)-\phi\left(\frac{n}{p_j}\right)-1\right)^2}.
     \end{align*}
    \end{itemize} \end{proof}
Next, we discuss the Sombor index of $W\Gamma(\mathbb{Z}_n)$ for some special values of $n$.
\begin{corollary}
    Let  $W\Gamma(\mathbb{Z}_n)$ be the weakly zero-divisor graph of $\mathbb{Z}_n$, $n \geq 2$. Then the following holds:
\begin{itemize}
     \item[(i)] For $n=p^{k}q$, where $p,q$ are primes and $k\geq1$, the Sombor index of  $W\Gamma(\mathbb{Z}_n)$ is
\begin{align*}
SO(W\Gamma(\mathbb{Z}_n))&=\frac{1}{\sqrt{2}}(p^{k-1}q-1)(p^{k-1}q-2)(p^{k-1}(p+q-1)-2)^2+(p^{k-1}q-1)p^{k-1}(p-1)\\&\sqrt{(p^{k-1}(p+q-1)-2)^2+(p^{k-1}q-1)^2}.
\end{align*}
   \item[(ii)]  For $n=pqr$,  where $p,q$ and $r$ are primes, the Sombor index of  $W\Gamma(\mathbb{Z}_n)$ is
   \begin{align*}
SO(W\Gamma(\mathbb{Z}_n))&=\frac{1}{\sqrt{2}}(p+q+r-3)((p+q+r-4)(p(q-1)+q(r-1)+r(p-1))-4)\\&+(p+q+r-3)(q-1)(r-1)\sqrt{(p(q-1)+q(r-1)+r(p-1))-4)^2+(pq+pr-p-1)^2} \\&+(p+q+r-3)(p-1)(r-1) \sqrt{(p(q-1)+q(r-1)+r(p-1))-4)^2+(pq+qr-q-1)^2}\\&+(p+q+r-3)(p-1)(q-1)\sqrt{(p(q-1)+q(r-1)+r(p-1))-4)^2+(pr+qr-r-1)^2}.
\end{align*}
 \end{itemize}
\begin{proof} 
\begin{itemize}
   \item[(i)]  For $n=p^kq $, the weakly zero-divisor graph of $\mathbb{Z}_n$ can be constructed as follows.
The proper divisors of $n$ are $p,p^2,\ldots, p^k, pq,p^2q,\ldots,p^{k-1}q,q$.  By Lemma \ref{inducedsubgraphequaltogamma}, we have   
 \[W\Gamma(\mathbb{Z}_n) = \Upsilon_{p^kq} [W\Gamma(\mathcal{A}_{p}), W\Gamma (\mathcal{A}_{p^2}), \cdots, W\Gamma(\mathcal{A}_{p^k}),  W\Gamma(\mathcal{A}_{pq}), W\Gamma (\mathcal{A}_{p^2q}), \cdots, W\Gamma(\mathcal{A}_{p^(k-1)q}), W\Gamma(\mathcal{A}_{q})].\]
       Note that by Lemma \ref{adjacenyofjoin}, for $1\leq i\leq k$, $W\Gamma(A_{p^i})=K_{\phi(p^{k-i}q)}$, $W\Gamma(A_{p^jq})=K_{\phi(p^{k-j})}$ for $1\leq j\leq k-1$ and $W\Gamma(A_{q})=\overline{K}_{\phi(p^{k})}$. By Lemma \ref{adjacenyofvertex}, $\Upsilon_{p^kq}$ is a complete graph on the set $\{ p,p^2,\ldots p^k, pq,p^2q,\ldots,p^{k-1}q,q\}$ . Consequently,
 \begin{align*}
           W\Gamma(\mathbb{Z}_n)&= K_{2k}[K_{\phi({p^{k-1}q})}, K_{\phi({p^{k-2}q})}, \ldots, K_{\phi({q})}, K_{\phi({p^{k-1}})},K_{\phi({p^{k-2}})},\ldots,K_{\phi({p})},\overline{K}_{\phi(p^{k})} ]
             \\&= K_{\left({\phi(p^{k-1}q)+\phi(p^{k-2}q)\ldots+\phi(pq)+ \phi(q)} +\phi(p^{k-1})+\phi(p^{k-2})+\ldots+\phi(p)\right)}\vee \overline{K}_{\phi(p^k)} \\&= K_{p^{k-1}q-1}\vee \overline{K}_{\phi(p^k)}.
      \end{align*}
      Observe that there are $\frac{1}{2}(p^{k-1}q-1)(p^{k-1}q-2)$ edges in  $K_{p^{k-1}q-1}$ and $\frac{1}{2}(p^{k-1}q-1)(p^{k-1}q-2)\phi(p^k) $  edges between $K_{p^{k-1}q-1}$ and $\overline{K}_{\phi(p^k)}$.
Thus, the Sombor index of $W\Gamma(\mathbb{Z}_n)$ is
\begin{align*}
    SO(W\Gamma(\mathbb{Z}_n)) &= \frac{1}{2}(p^{k-1}q-1)(p^{k-1}q-2)\sqrt{(p^{k-1}(p+q-1)-2)^2+(p^{k-1}(p+q-1)-2)^2}+ (p^{k-1}q-1)\phi(p^k)\\&\sqrt{(p^{k-1}(p+q-1)-2)^2+(p^{k-1}q-1)^2}\\&= \frac{1}{\sqrt{2}}(p^{k-1}q-1)(p^{k-1}q-2)(p^{k-1}(p+q-1)-2)+(p^{k-1}q-1)p^{k-1}(p-1)\\&\sqrt{(p^{k-1}(p+q-1)-2)^2+(p^{k-1}q-1)^2}.
\end{align*}
\item[(ii)] Let $n=pqr $,  where $p,q$ and $r$ are prime. Then by Lemma \ref{adjacenyofvertexinjoin}, the weakly zero-divisor graph of $\mathbb{Z}_n$ is given by
\[W\Gamma(\mathbb{Z}_n) = \Upsilon_{pqr} [W\Gamma(\mathcal{A}_{p}), W\Gamma (\mathcal{A}_{q}), W\Gamma(\mathcal{A}_{r}),  W\Gamma(\mathcal{A}_{pq}), W\Gamma (\mathcal{A}_{pr}), W\Gamma(\mathcal{A}_{qr})].\] 
By Lemma \ref{adjacenyofjoin}, $W\Gamma (\mathcal{A}_{p})=\overline{K}_{\phi(\frac{n}{p})}$, $W\Gamma (\mathcal{A}_{q})=\overline{K}_{\phi(\frac{n}{q})}$, $W\Gamma (\mathcal{A}_{r})=\overline{K}_{\phi(\frac{n}{r})}$, $W\Gamma (\mathcal{A}_{pq})={K}_{\phi(\frac{n}{pq})}$, $W\Gamma (\mathcal{A}_{pr})={K}_{\phi(\frac{n}{pr})}$, $W\Gamma (\mathcal{A}_{qr})={K}_{\phi(\frac{n}{qr})}$, $|A_{p}|=\phi(\frac{n}{p})$, $|A_{q}|=\phi(\frac{n}{q})$, $|A_{r}|=\phi(\frac{n}{p})$, $|A_{pq}|=\phi(\frac{n}{pq})$, $|A_{pr}|=\phi(\frac{n}{pr})$, $|A_{qr}|=\phi(\frac{n}{qr})$ and  by Lemma \ref{adjacenyofvertex}, the graph $\Upsilon_{pqr}$ is a complete on the  set $\{p,q,r,pq,pr,qr\}$. It follows that
\begin{align*}
    W\Gamma(\mathbb{Z}_n) &= K_6 [\overline{K}_{\phi(\frac{n}{p})},\overline{K}_{\phi(\frac{n}{q})}), \overline{K}_{\phi(\frac{n}{r})}, {K}_{\phi(\frac{n}{pq})}, {K}_{\phi(\frac{n}{pr})}), {K}_{\phi(\frac{n}{qr})}]\\&= K_{\phi(\frac{n}{pq})+\phi(\frac{n}{pr})+\phi(\frac{n}{qr})}\vee K_3[\overline{K}_{\phi(\frac{n}{p})},\overline{K}_{\phi(\frac{n}{q})}), \overline{K}_{\phi(\frac{n}{r})}]\\&=K_{p+q+r-3}\vee K_3[\overline{K}_{\phi(\frac{n}{p})},\overline{K}_{\phi(\frac{n}{q})}), \overline{K}_{\phi(\frac{n}{r})}].                  
    \end{align*}
 Let $\{v_1,v_2,\ldots,v_{p+q+r-3}, x_1,x_2,\ldots,x_{\phi(\frac{n}{p})},  y_1,y_2,\ldots,y_{\phi(\frac{n}{q})}, z_1,z_2,\ldots,z_{\phi(\frac{n}{r})}\}$ is the vertex labelling of graph $ W\Gamma(\mathbb{Z}_n)$, where $v_i$ are the vertices of $K_{p+q+r-3}$, $x_j$ are the vertices of $\overline{K}_{\phi(\frac{n}{p})}$, $y_k$ are the vertices of $\overline{K}_{\phi(\frac{n}{q})}$ and $z_l$ are the vertices of $\overline{K}_{\phi(\frac{n}{r})}$. Observe that  $\frak{d}_{v_i}=p+q+r-3+\phi(pq)+\phi(pr)+\phi(pq)-1=p(q-1)+q(r-1)+r(p-1)-4$, $\frak{d}_{x_j}=p+q+r-3+\phi(pr)+\phi(pq)= pr+pq-p-1$,  $\frak{d}_{y_k}=p+q+r-3+\phi(qr)+\phi(pq)= pq+qr-q-1$ and $\frak{d}_{z_l}=p+q+r-3+\phi(qr)+\phi(pr)= pr+qr-r-1$. 
   Therefore, the Sombor index of  $W\Gamma(\mathbb{Z}_n)$ is  given by

 \begin{align*}
    SO(W\Gamma(\mathbb{Z}_n)) &=\sum \limits_{v_i\sim v_{j}}\sqrt{\frak{d}_
    {v_i}^2+\frak{d}_{v_{j}}^2}+\sum \limits_{v_i\sim x_{j}}\sqrt{\frak{d}_
    {v_i}^2+\frak{d}_{x_{j}}^2}+ \sum \limits_{v_i\sim y_{k}}\sqrt{\frak{d}_
    {v_i}^2+\frak{d}_{y_{k}}^2} +\sum \limits_{v_i\sim z_{l}}\sqrt{\frak{d}_
    {v_i}^2+\frak{d}_{z_{l}}^2}\\&+ \sum \limits_{x_j\sim y_{k}}\sqrt{\frak{d}_{x_{j}}^2+\frak{d}_{y_{k}}^2}+\sum \limits_{x_j\sim z_{l}}\sqrt{\frak{d}_{x_{j}}^2+\frak{d}_{z_{l}}^2}+ \sum \limits_{y_k\sim z_{l}}\sqrt{\frak{d}_{y_{k}}^2+\frak{d}_{z_{l}}^2} \\&= \frac{(p+q+r-3)(p+q+r-4)}{2}\sqrt{((p(q-1)+q(r-1)+r(p-1))-4)^2+((p(q-1)+q(r-1)+r(p-1))-4)^2}\\&+ (p+q+r-3)\phi(qr)\sqrt{((p(q-1)+q(r-1)+r(p-1))-4)^2+(pq+pr-p-1)^2}+ (p+q+r-3)\phi(pr)\\& \sqrt{((p(q-1)+q(r-1)+r(p-1))-4)^2+(pq+qr-q-1)^2}+ (p+q+r-3)\phi(pq)\\&\sqrt{((p(q-1)+q(r-1)+r(p-1))-4)^2+(pr+qr-r-1)^2}+\phi(qr)\phi(pr)\\&\sqrt{(pq+pr-p-1)^2+(pq+qr-q-1)^2}+\phi(qr)\phi(pq)\sqrt{(pq+pr-p-1)^2+(pr+qr-r-1)^2}\\&+\phi(pr)\phi(qr)\sqrt{(pq+qr-q-1)^2+(pr+qr-r-1)^2}\\&=\frac{1}{\sqrt{2}}(p+q+r-3)((p+q+r-4)(p(q-1)+q(r-1)+r(p-1))-4)+(p+q+r-3)(q-1)(r-1)\\&\sqrt{((p(q-1)+q(r-1)+r(p-1))-4)^2+(pq+pr-p-1)^2}+(p+q+r-3)(p-1)(r-1)\\& \sqrt{((p(q-1)+q(r-1)+r(p-1))-4)^2+(pq+qr-q-1)^2}+(p+q+r-3)(p-1)(q-1)\\&\sqrt{((p(q-1)+q(r-1)+r(p-1))-4)^2+(pr+qr-r-1)^2}.
 \end{align*}
 \end{itemize}
\end{proof}
\end{corollary}
\section{Sombor Spectrum of the Weakly Zero-divisor Graphs of Integer Modulo Ring}
In this section, we first recall some definitions and established results that will be useful for later use. Moreover, we determine the Sombor eigenvalues of the weakly zero-divisor graph $W\Gamma(\mathbb{Z}_n)$.
\begin{lemma}\cite{pirzada2021sombor}\label{eigenvalueofclicque}
    Let $G$ be a connected graph of order $n$ and $S= \{v_1, v_2,\ldots,v_\alpha \}$
be the set of vertices of $G$ satisfying the condition $N(v_i)\setminus S = N(v_j )\setminus S$ for each
$1 \leq i, j \leq {\alpha} $. Then, the following holds.
\begin{itemize}
\item[(i)]  If $S$ is an independent set, then $0$ is the Sombor eigenvalue of $G$ with multiplicity at least $\alpha -1$.
\item[(ii)]  If $S$ is a clique, then $ -d\sqrt{2}$ is the Sombor eigenvalue of $G$ with multiplicity at
least $\alpha -1$, where $d$ is the degree of $v_i$.
\end{itemize}
\end{lemma}
Consider an n × n matrix
\[M =  
 \displaystyle \begin{bmatrix}
 
A_{1,1} &A_{1,2} &\cdots& A_{1,s-1}& A_{1,s}\\
A_{2,1} & A_{2,2} &\cdots&  A_{2,s-1} & A_{2,s}\\
\vdots & \vdots &  \vdots & \vdots & \vdots &  \\
\vdots & \vdots &  \vdots & \vdots & \vdots &  \\
\vdots & \vdots &  \vdots & \vdots & \vdots &  \\ 
A_{s-1,1} & A_{s-1,2} &\cdots& A_{s-1,s-1}& A_{s-1,s}\\
A_{s,1} &A_{s,2} &\cdots& A_{s,s-1} &A_{s,s}\\
\end{bmatrix},\]\\
whose rows and columns are partitioned according to a partition $P =
\{P_1, P_2,\ldots,P_s\}$ of the set $X = \{1, 2,\ldots,n\}$. The quotient matrix $Q = (q_{ij} )$ (see [5]) is an $s \times s$ matrix, where $q_{ij}$th entry is the average row (column) sum of the
block $A_{ij}$ of $M$. The partition $P$ is said to be equitable if row (column) sum of each
block $A_{i,j}$ is some constant and in such case $Q$ is known as the equitable quotient
matrix. The next result gives a relation between the eigenvalues of $M$ and the eigenvalues of $Q$.
\begin{theorem}\cite{brouwer2011spectra}
     Let $M$ be a matrix of order $n$ and $Q$ be its quotient matrix.
Then the following hold.
\begin{itemize}
    \item[(i)]  If the partition $P$ of $X$ of matrix $M$ is not equitable, then the eigenvalues of $Q$ interlace the eigenvalues of $M$.
\item[(ii)] If the partition $P$ of $X$ of matrix $M$ is equitable, then each of the eigenvalue of
$Q$ is the eigenvalue of $M$.
\end{itemize}
\end{theorem}


The following result determines the Sombor spectrum of the weakly zero-divisor graph of $\mathbb{Z}_n$.
 \begin{theorem}\label{somberspectrum}
     Let $W\Gamma(\mathbb{Z}_n)$ be a weakly zero-divisor graph of order $n \geq 2$. Then the following holds:
 \begin{itemize}
     \item[(i)] If $n={p_1}^{n_1}{p_2}^{n_2}{p_3}^{n_3}\cdots{p_k}^{n_k}$, where $n_i\geq2$, then the Sombor spectrum of $W(\Gamma(\mathbb{Z}_n))$ is 
    \begin{center}

$\left\{(-\sqrt{2}(n-\phi(n)-2))^{[n-\phi(n)-2]}, \sqrt{2}(n-\phi(n)-1)^2\right\}$.
\end{center} 
\item[(ii)]  If $n = p_1p_2\cdots p_m{q_1}^{k_1}{q_2}^{k_2}\cdots{q_r}^{k_r} $, where $k_i\geq 2,m\geq1$ and $r+m\geq2$, then the Sombor spectrum of $W\Gamma(\mathbb{Z}_n)$ consists of the eigenvalues $-\sqrt{2}(n-\phi(n)-2)$ with multiplicity $(n-\phi(n)-\sum\limits_{i=1}^m\phi (\frac{n}{p_i})-2)$, the eigenvalue $0$ with multiplicity $\sum\limits_{i=1}^m\phi(\frac{n}{p_i})-m$, and the eigenvalues of the following equitable quotient matrix:
\begin{equation}
 \displaystyle \begin{bmatrix}
c_1 &  b_1 \phi(\frac{n}{p_1}) & b_2  \phi(\frac{n}{p_2})&\cdots &
\cdots & b_{m-1}  \phi(\frac{n}{p_{m-1}}) &  \phi(\frac{n}{p_m})\\
 b_1 k & 0 & \chi_{12} &\cdots &\cdots& \chi_{1{m-1}}& \chi_{1m}\\
b_2  k & \chi_{21} & 0&\cdots &\cdots&\chi_{2{m-1}}& \chi_{2m}\\
\vdots&\vdots&\vdots&\vdots&\vdots&\vdots&\vdots\\
\vdots&\vdots&\vdots&\vdots&\vdots&\vdots&\vdots\\
b_{m-1} k& \chi_{{m-1}1}& \chi_{{m-1}2}&\cdots&\cdots&  0 & \chi_{{m-1}{m}}\\
b_{m} k& A_{{m}1}& A_{{m}2}&\cdots&\cdots& A_{{m}{m-1}} &0_{\phi(\frac{n}{p_{m}})\times\phi(\frac{n}{p_m})} \\
\end{bmatrix},
\end{equation}
 where $c_1=\sqrt{2}(k-1)(n-\phi(n)-2)$, $\chi_{ij}=\phi(\frac{n}{p_j})\sqrt{(n-\phi(n)-\phi(\frac{n}{p_i})-1)^2+(n-\phi(n)-\phi(\frac{n}{p_j})-1)^2}$,  $k=n-\phi(n)-\sum\limits_{i=1}^k\phi (\frac{n}{p_i})-1$ and $b_i=\sqrt{(n-\phi(n)-2)^2+(n-\phi(n)-\phi(\frac{n}{p_j})-1)^2}.$
\end{itemize} 
\end{theorem}
\begin{proof}
\begin{itemize}
    \item[(i)] Let $n={p_1}^{n_1}{p_2}^{n_2}{p_3}^{n_3}\cdots{p_k}^{n_k}$, where $n_i\geq2$. Then the weakly zero-divisor graph is complete. Therefore, the only Sombor eigen values are $-\sqrt{2}(n-\phi(n)-2)$ with multiplicity $(n-\phi(n)-2)$ and $\sqrt{2}(n-\phi(n)-1)^2$ with multiplicity $1$. 
\item[(ii)]  Let $n = p_1p_2\cdots p_k{q_1}^{k_1}{q_2}^{k_2}\cdots{q_r}^{k_r} $, where $k_i\geq 2$, $k\geq1$ and $r+m\geq2$. Let $D=\{d_1,d_2,\ldots, d_{\tau(n)-2}\}$ be the set of proper divisor of $n$. Thus, by  Lemma \ref{inducedsubgraphequaltogamma}, the weakly zero-divisor graph $W\Gamma(\mathbb{Z}_n)$  is given by
        \[ W\Gamma(\mathbb{Z}_n) = \Upsilon_n [W\Gamma(\mathcal{A}_{d_1}), W\Gamma (\mathcal{A}_{d_2}), \ldots, W\Gamma(\mathcal{A}_{d_m}),  W\Gamma(\mathcal{A}_{d_{m+1}}), W\Gamma (\mathcal{A}_{d_{m+2}}), \ldots, W\Gamma(\mathcal{A}_{d_{\tau(n)-2}})].\]
 By Lemma \ref{adjacenyofvertex}, the graph $\Upsilon_{n}$ is complete on the vertex set $D$. Now, let $A=\{p_1,p_2,\ldots p_m\}$, then by Lemma \ref{adjacenyofjoin}, for each $d_i\in A$, we have $W\Gamma(\mathcal{A}_{d_{i}}) = \overline{K}_{\phi \left(\frac{n}{d_i}\right)}$ and $W\Gamma(\mathcal{A}_{d_j}) =  {K}_{\phi\left(\frac{n}{d_j}\right)}$, where $d_j\notin A$.
   In the view of Lemma \ref{adjacenyofjoin} and Corollary \ref{partitionofcozerodivisorgraphisomorphic}, we get       
 \[W\Gamma(\mathbb{Z}_n)= K_{\left(n-\phi(n)-\sum\limits_{i=1}^k\phi\left( \frac{n}{p_i}\right)-1\right)} \vee K_m[\overline{K}_{\phi(\frac{n}{p_1})}, \overline{K}_{\phi(\frac{n}{p_2})},\ldots, \overline{K}_{\phi(\frac{n}{p_k})}].\] 
 Let $\{ v_1, v_2, \ldots, v_{\left(n-\phi(n) - \sum\limits_{i=1}^{m} \phi \left( \frac{n}{p_i} \right)
 -1\right)}, u_{11}, u_{12}, \ldots, u_{1\phi \left( \frac{n}{p_1} \right)}, u_{21}, u_{22}, \ldots, u_{2\phi \left( \frac{n}{p_2} \right)}, \ldots, u_{m1}, u_{m2}, \ldots, u_{m\phi \left( \frac{n}{p_m} \right)} \}$ be the vertex labeling of $W\Gamma(\mathbb{Z}_n)$, where $v_i$’s are vertices of 
$K_{\left(n-\phi(n)-\sum\limits_{i=1}^k\phi( \frac{n}{p_i})\right)}$, $u_{ij}$ 
are the vertices of $\overline{K}_{\phi\left(\frac{n}{p_j}\right)}$ for $i= 1,2,\ldots,m$ and $j= 1,2,\ldots,m$.
Note that $\frak{d}_{v_1} = \frak{d}_{v_1} = \cdots = \frak{d}_{v_{\left(n-\phi(n)-\sum\limits_{i=1}^m\phi (\frac{n}{p_i})\right)}}= n-\phi(n)-\sum\limits_{i=1}^m\phi (\frac{n}{p_i})-2-\sum\limits_{i=1}^m\phi (\frac{n}{p_i})=n-\phi(n)-2$, $\frak{d}_{u_{11}} =\frak{d}_{u_{12}}=\cdots=\frak{d}_{u_{1\phi(\frac{n}{p_1})}}= n-\phi(n)-\sum\limits_{i=1}^m\phi (\frac{n}{p_i})-1+\sum\limits_{i=1}^m\phi (\frac{n}{p_i})-\phi(\frac{n}{p_1})=n-\phi(n)-\phi(\frac{n}{p_1})-1$, $\frak{d}_{u_{21}}=\frak{d}_{u_{22}}=\cdots=\frak{d}_{u_{2\phi(\frac{n}{p_2})}}=n-\phi(n)-\sum\limits_{i=1}^m\phi (\frac{n}{p_i})-1+\sum\limits_{i=1}^m\phi (\frac{n}{p_i})-\phi(\frac{n}{p_2})=n-\phi(n)-\phi(\frac{n}{p_2})-1$, \ldots, $\frak{d}_{u_{m1}} =\frak{d}_{u_{m2}}=\cdots=\frak{d}_{u_{m\phi(\frac{n}{p_2})}}=n-\phi(n)-\sum\limits_{i=1}^m\phi (\frac{n}{p_i})-1+\sum\limits_{i=1}^m\phi (\frac{n}{p_i})-\phi(\frac{n}{p_m})=n-\phi(n)-\phi(\frac{n}{p_m})-1$. With this labeling, the Sombor matrix of $W\Gamma(\mathbb{Z}_n)$ is given by
\begin{equation}
 \displaystyle \begin{bmatrix}
C_{k} &  a_1 \mathcal{J}_{k\times \phi(\frac{n}{p_1})} & a_2 \mathcal{J}_{k\times \phi(\frac{n}{p_2})}&\cdots &
\cdots & a_{m-1} \mathcal{J}_{k\times \phi(\frac{n}{p_{m-1}})} & a_m \mathcal{J}_{k\times \phi(\frac{n}{p_m})}\\
a_1\mathcal{J}_{\phi(\frac{n}{p_{1}})\times k }& 0_{\phi(\frac{n}{p_1})\times \phi(\frac{n}{p_1})}& A_{12} &\cdots &\cdots& A_{1{m-1}}& A_{1m}\\
a_2 \mathcal{J}_{ \phi(\frac{n}{p_2})\times k} & A_{21} & 0_{\phi(\frac{n}{p_2})\times\phi(\frac{n}{p_2})}&\cdots &\cdots&A_{2{m-1}}& A_{2m}\\
\vdots&\vdots&\vdots&\vdots&\vdots&\vdots&\vdots\\
\vdots&\vdots&\vdots&\vdots&\vdots&\vdots&\vdots\\
a_{m-1} \mathcal{J}_{ \phi(\frac{n}{p_{m-1}})\times k}& A_{{m-1}1}& A_{{m-1}2}&\cdots&\cdots&  0_{\phi(\frac{n}{p_{m-1}})\times\phi(\frac{n}{m-1})} & A_{{m-1}{m}}\\
a_{m} \mathcal{J}_{ \phi(\frac{n}{p_{m}})\times k}& A_{{m}1}& A_{{m}2}&\cdots&\cdots& A_{{m}{m-1}} &0_{\phi(\frac{n}{p_{m}})\times\phi(\frac{n}{p_m})} \\
\end{bmatrix},
    \end{equation}
Where $\mathcal{J}$ is a matrix of all ones, $0$ is the zero matrix, $C_k=\sqrt{2}(n-\phi(n)-1)(\mathcal{J}_k-\mathcal{I}_k)$, where $k=n-\phi(n)-\sum\limits_{i=1}^k\phi (\frac{n}{p_i})-1$.  For $1\leq i\leq m$,  $a_i=\sqrt{(n-\phi(n)-2)^2+(n-\phi(n)-\phi(\frac{n}{p_i})-1)^2}$ and $A_{ij}=\sqrt{(n-\phi(n)-\phi(\frac{n}{p_i})-1)^2+(n-\phi(n)-\phi(\frac{n}{p_j})-1)^2} \mathcal{J}_{\phi(\frac{n}{p_i})\times\phi(\frac{n}{p_j})}$.
Since,
each vertex  $v_i’s$ of $K_k$ share the same neighborhood and so by Lemma \ref{eigenvalueofclicque},
$-\sqrt{2}(n-\phi(n)-2)$ is the Sombor eigenvalue of $W\Gamma(\mathbb{Z}_n)$ with multiplicity $(n-\phi(n)-\sum\limits_{i=1}^k\phi (\frac{n}{p_i})-2)$. Also, each of $\overline{K}_{(\frac{n}{p_i}})$ form an independent set and
share a common vertex set in $W\Gamma(\mathbb{Z}_n)$, by Lemma \ref{eigenvalueofclicque}, $0$ is the Sombor eigenvalue of $W\Gamma(\mathbb{Z}_n)$ with multiplicity $\sum\limits_{i=1}^m\phi(\frac{n}{p_i})-m$.
The other $m+1$ Sombor eigenvalues of $W\Gamma(\mathbb{Z}_n)$ are the eigenvalues of the equitable quotient matrix of Matrix (2) and are given in (1).
\end{itemize}
\end{proof}
Next, we discuss the Sombor spectrum of $W\Gamma(\mathbb{Z}_n)$ for some values of $n$.
\begin{corollary}
    Let $\Gamma(\mathbb{Z}_n)$ be the weakly zero-divisor graph of $\mathbb{Z}_n$. Then the following hold.
    \begin{itemize}
    \item[(i)]  If $n = p^k q$, where $p$ and $q$ are prime and $k\geq 2$ is a positive integer, then the Sombor spectrum of $\Gamma(\mathbb{Z}_n)$ is:
\[\left\{0^{[\phi(p^k)-1]}, \, (-\sqrt{2}(p^{k-1}q+\phi(p^k)-1))^{[p^{k-1}q - 1]}, \, \frac{\sqrt{2}(p^{k-1}q + \phi(p^k) - 1)(p^{k-1}q - 2) \pm \sqrt{\lambda}}{2}\right\} \]
where $\lambda= 2(p^{k-1}q + \phi(p^k) - 1)^2(p^{k-1}q - 2)^2 + 4\phi(p^k)(p^{k-1}q - 1)^2(p^{k-1}q + \phi(p^k) - 1)^2$.
 \item[(ii)]  If $n=pqr$, then  the Sombor spectrum of $\Gamma(\mathbb{Z}_n)$  consists of the eigenvalues $-\sqrt{2}(p(q-1)+q(r-1)+r(p-1)-4)$ with multiplicity $p+q+r-4$, the eigenvalue $0$ with multiplicity $(p-1)(q-1)+(p-1)(r-1)+(q-1)(r-1)-4$, and the eigenvalues of the following equitable quotient matrix (3):
 \end{itemize}
 \end{corollary}
\begin{proof}
\begin{enumerate}
    \item[(i)] Let $p^k q$, where $p$ and $q$ are distinct primes. Then by definition, the weakly zero divisor graph of  $\mathbb{Z}_n$ is given by
       \begin{align*}
           W\Gamma(\mathbb{Z}_n) &= K_{\left({\phi(p^{k-1}q)+\phi(p^{k-2}q)\ldots+\phi(pq)+\phi(q)} +\phi(p^{k-1})+\phi(p^{k-2})+\ldots+\phi(p)\right)}\vee \overline{K}_{\phi(p^k)} \\&= K_{p^{k-1}q-1}\vee \overline{K}_{\phi(p^k)}.
      \end{align*}

Now, by Theorem 3.3, the Sombor spectrum of $W(\Gamma(\mathbb{Z}_n))$ consists of the eigenvalue $-\sqrt{2}(p^{k-1}q+\phi(p^k)-1)$ with multiplicity $p^{k-1}q - 1$, the eigenvalue $0$ with multiplicity $\phi(p^k)-1$, and the eigenvalues of the following equitable quotient matrix:
 \[  \displaystyle \begin{bmatrix}
\sqrt{2}(p^{k-1}q+\phi(p^k)-1)(p^{k-1}q - 2) &  \phi(p^k) \sqrt{({p^{k-1}q+\phi(p^k)-1})^2+{(p^{k-1}q - 1)}^2}   \\(p^{k-1}q - 1)\sqrt{({p^{k-1}q+\phi(p^k)-1})^2 +{(p^{k-1}q - 1)}^2} & 0\\
\end{bmatrix}.\]\\
Observe that the eigenvalues of the above matrix are
\[\frac{\sqrt{2}(p^{k-1}q + \phi(p^k) - 1)(p^{k-1}q - 2) \pm \sqrt{2(p^{k-1}q + \phi(p^k) - 1)^2(p^{k-1}q - 2)^2 + 4\phi(p^k)(p^{k-1}q - 1)^2(p^{k-1}q + \phi(p^k) - 1)^2}}{2}.\]
        \item[(ii)]   For $n=pqr$, where $p,q$ and $r$ are distinct primes, then by definition, the weakly zero divisor graph of  $\mathbb{Z}_n$ is
    \begin{align*}
W\Gamma(\mathbb{Z}_n) &= \Upsilon_{pqr} [W\Gamma(\mathcal{A}_{p}), W\Gamma (\mathcal{A}_{q}), W\Gamma(\mathcal{A}_{r}),  W\Gamma(\mathcal{A}_{pq}), W\Gamma (\mathcal{A}_{pr}), W\Gamma(\mathcal{A}_{qr})\\& = K_6 [\overline{K}_{\phi(\frac{n}{p})},\overline{K}_{\phi(\frac{n}{q})}), \overline{K}_{\phi(\frac{n}{r})}, {K}_{\phi(\frac{n}{pq})}, {K}_{\phi(\frac{n}{pr})}), {K}_{\phi(\frac{n}{qr})}]\\&= K_{\phi(\frac{n}{pq})+\phi(\frac{n}{pr})+\phi(\frac{n}{qr})}\vee K_3[\overline{K}_{\phi(\frac{n}{p})},\overline{K}_{\phi(\frac{n}{q})}), \overline{K}_{\phi(\frac{n}{r})}]\\&=K_{p+q+r-3}\vee K_3[\overline{K}_{\phi(\frac{n}{p})},\overline{K}_{\phi(\frac{n}{q})}), \overline{K}_{\phi(\frac{n}{r})}].                  
 \end{align*}
Using  Theorem \ref{somberspectrum}, the Sombor spectrum of $W(\Gamma(\mathbb{Z}_n))$  consists of the eigenvalue $-\sqrt{2}(p(q-1)+q(r-1)+r(p-1)-4)$ with multiplicity $p+q+r-4$, the eigenvalue $0$ with multiplicity $(p-1)(q-1)+(p-1)(r-1)+(q-1)(r-1)-4$, and the eigenvalues of the following equitable quotient matrix:
\begin{equation}
   \displaystyle \begin{bmatrix}
(p+q+r-4) \sqrt{2}d_1&   (p-1)(q-1)\sqrt{{d_1}^2+{d_2}^2} & (p-1)(r-1)\sqrt{{d_1}^2+{d_3}^2} &(q-1)(r-1) \sqrt{{d_1}^2+{d_4}^2} \\
 (p+q+r-3)\sqrt{{d_1}^2+{d_2}^2} & 0 &  (p-1)(r-1) \sqrt{{d_2}^2+{d_3}^2} &(p-1)(q-1) \sqrt{{d_2}^2+{d_2}^2}  \\ (p+q+r-3)
\sqrt{{d_1}^2+{d_3}^2} & (p-1)(q-1)\sqrt{{d_3}^2+{d_2}^2}  & 0&(q-1)(r-1)\sqrt{{d_3}^2+{d_4}^2} \\
(p+q+r-3)\sqrt{{d_1}^2+{d_4}^2} &(p-1)(q-1)\sqrt{{d_4}^2+{d_2}^2} &(p-1)(r-1)\sqrt{{d_4}^2+{d_3}^2} &0\\
\end{bmatrix},
\end{equation}
where $d_1=p(q-1)+q(r-1)+r(p-1)-4$, $d_2=pq+pr-p-1$,  $d_3=pq+qr-q-1$ and  $d_4=pr+pr-r-1$.
\end{enumerate}
    \end{proof}
Now, we obtain a lower bound on Sombor energy of $ W\Gamma(\mathbb{Z}_n)$.
\begin{theorem}
     Let $n = p_1p_2\ldots p_m{q_1}^{k_1}{q_2}^{k_2}\ldots{q_r}^{k_r} $, where $k_i\geq 2,m\geq1$ and $r+m\geq2$. Then the Sombor energy of the weakly zero-divisor graph $ W\Gamma(\mathbb{Z}_n)$ is given by

\[
ESO(G) \geq \sqrt{2}(n-\phi(n)-2) \left(n-\phi(n)-\sum\limits_{i=1}^m\phi \left(\frac{n}{p_i}\right)-2\right).
\]

\begin{proof}
    
 From Theorem 3.3, $G$ has a nonzero Sombor eigenvalue $-\sqrt{2}(n-\phi(n)-2)$ with multiplicity $(n-\phi(n)-\sum\limits_{i=1}^m\phi\left (\frac{n}{p_i}\right)-2)$, and the other nonzero eigenvalues are the eigenvalues of the matrix in (3.8). Thus, by the definition of Sombor energy, we have

\[
ESO(G) = \sum_{i=1}^{\left(n-\phi(n)-\sum\limits_{i=1}^m\phi \left(\frac{n}{p_i}\right)-2\right)} \left| -\sqrt{2}(n-\phi(n)-2) \right| + E(M)  \geq \sqrt{2}(n-\phi(n)-2)  \left(n-\phi(n)-\sum\limits_{i=1}^m\phi \left(\frac{n}{p_i}\right)-2\right).
\]
\end{proof}
\end{theorem}
\begin{proposition}
Let $W(\Gamma(\mathbb{Z}_n))$ be the weakly zero-divisor graph  of $\mathbb{Z}_n$. Then the following hold,
\begin{itemize}
    \item[(i)] If $n=p^k$, where $p$ is a prime, then the Sombor energy of $W(\Gamma(\mathbb{Z}_{p^{k}}))$ is
    \[
    E_{SO}(G) =  \sqrt{2}(p^{k-1}-2)^2+ \sqrt{2}(p^{k
-1}-1)^2.
    \]
    \item[(ii)] If $n = p^kq$, where $p$ and $q$ are  prime and $k \geq 2$ is a positive integer, then the Sombor spectrum of $W(\Gamma(\mathbb{Z}_n))$ is
    \[
    E_{SO}(\Gamma(\mathbb{Z}_n)) = \sqrt{2}(p^{k-1}q+\phi(p^k)-1) (p^{k-1}q - 1) + \sqrt{\lambda},
    \]
  where  $\lambda= 2(p^{k-1}q + \phi(p^k) - 1)^2(p^{k-1}q - 2)^2 + 4\phi(p^k)(p^{k-1}q - 1)^2(p^{k-1}q + \phi(p^k) - 1)^2$.
   
\end{itemize}
\end{proposition}

\begin{proof}
If $n = p^k$, where $p$ is prime, then $W\Gamma(\mathbb{Z}_n) \cong K_{p^k-\phi(p^k)-1}$,
\begin{align*}
ESO(G) &= \sum_{i=1}^{p^k-\phi(p^k)-2} \left| -\sqrt{2}(n-\phi(p^k)-2) \right| + \left| \sqrt{2}(p^k-\phi(p^k)-1)^2 \right| \\
&= (p^k-\phi(p^k)-2)( \sqrt{2}(p^k-\phi(p^k)-2)+ \sqrt{2}(p^k-\phi(p^k)-1)^2\\ &= \sqrt{2}(p^{k-1}-2)^2+ \sqrt{2}(p^{k
-1}-1)^2
\end{align*}

For $n = p^kq$, where $p$ is prime and $k \geq 2$ is a positive integer, by Corollary 3.4 and the definition of Sombor energy, we have
 
   \[ E_{SO}(\Gamma(\mathbb{Z}_n)) = \sqrt{2}(p^{k-1}q+\phi(p^k)-1) (p^{k-1}q - 1) + \sqrt{\lambda},\]
    where  $\lambda$ is defined above.
\end{proof}

\textbf{Acknowledgement:} The first author gratefully acknowledges Birla Institute of Technology and Science (BITS) Pilani, India, for providing financial support.

\vspace{.3cm}
\textbf{Conflicts of interest/Competing interests}: There is no conflict of interest regarding the publishing of this paper.


\end{document}